\documentclass[11pt,reqno]{amsart}
\usepackage{enumerate}
\usepackage{amsrefs}
\usepackage{amsfonts, amsmath, amssymb, amscd, amsthm, bm, cancel}
\usepackage{mathrsfs}
\usepackage{url}
\usepackage{graphicx}
\usepackage[
linktocpage=true,colorlinks,citecolor=magenta,linkcolor=blue,urlcolor=magenta]{hyperref}
\usepackage{multicol}
\usepackage{comment}
\usepackage{amsrefs}
\newtheorem{thm}{Theorem}[section]
\newtheorem*{thmA}{Theorem A}
\newtheorem*{thmB}{Theorem B}
\newtheorem*{thmC}{Theorem C}
\newtheorem{cor}[thm]{Corollary}

\newtheorem{lem}[thm]{Lemma}

\newtheorem{prop}[thm]{Proposition}

\theoremstyle{definition}

\newtheorem*{ack}{Acknowledgments}
\theoremstyle{remark}

\numberwithin{equation}{section}
\allowdisplaybreaks

\renewcommand{\(}{\left(}
\renewcommand{\)}{\right)}
\renewcommand{\~}{\widetilde}
\renewcommand{\-}{\overline}

\renewcommand{\a}{\alpha}

\renewcommand{\d}{\delta}
\newcommand{\e}{\varepsilon}
\renewcommand{\k}{\kappa}
\renewcommand{\l}{\lambda}

\newcommand{\D}{\Delta}

\newcommand{\s}{\sigma}

\newcommand{\G}{\Gamma}

\renewcommand{\l}{\lambda}

\newcommand{\ra}{\rightarrow}

\newcommand{\vol}{\operatorname{vol}}
\newcommand{\Vol}{\operatorname{Vol}}

\newcommand{\divv}{\operatorname{div}}

\begin{document}
\title[A Heintze-Karcher type inequality in hyperbolic space]{A Heintze-Karcher type inequality in hyperbolic space}
\author[Y. Hu]{Yingxiang Hu}
\address{School of Mathematical Sciences, Beihang University, Beijing 100191, P.R. China}
\email{\href{mailto:huyingxiang@buaa.edu.cn}{huyingxiang@buaa.edu.cn}}
\author[Y. Wei]{Yong Wei}
\author[T. Zhou]{Tailong Zhou}
\address{School of Mathematical Sciences, University of Science and Technology of China, Hefei 230026, P.R. China}
\email{\href{mailto:yongwei@ustc.edu.cn}{yongwei@ustc.edu.cn}}
\email{\href{mailto:ztl20@ustc.edu.cn}{ztl20@ustc.edu.cn}}


\begin{abstract}
In this paper, we prove a new Heintze-Karcher type inequality for shifted mean convex hypersurfaces in hyperbolic space. As applications, we prove an Alexandrov type theorem for closed embedded hypersurfaces with constant shifted $k$th mean curvature in hyperbolic space. Furthermore, a uniqueness result for $h$-convex hypersurfaces satisfying certain curvature equations is obtained.
\end{abstract}

\maketitle

\section{Introduction}\label{sec:1}
The classical Alexandrov's theorem \cite{Alexandrov1956} states that any closed, embedded hypersurface in Euclidean space with constant mean curvature is a round sphere. Alexandrov's method is based on the maximum principle for elliptic equations. This result essentially improves previous results due to S\"uss \cite{Suss1952} and Hsiung \cite{Hsiung1954}, which require the additional assumption of convexity or star-shapedness on the hypersurface. Later, Reilly \cite{Reilly1977} provided a new proof of the Alexandrov's theorem by using his famous integral formula. Another proof of Alexandrov's theorem was also provided by Montiel and Ros \cite{Montiel-Ros1991}, by combining the Minkowski formulae with the following Heintze-Karcher inequality.
\begin{thmA}
	Let $\Omega$ be a bounded domain with smooth boundary $\Sigma=\partial \Omega$ in Euclidean space $\mathbb R^{n+1}$. Assume that the mean curvature $H$ of $\Sigma$ is positive, then
	\begin{align}\label{s1:Heintze-Karcher-original}
	\int_{\Sigma} \frac{1}{H} d\mu  \geq \frac{n+1}{n}\Vol(\Omega).
	\end{align}
	Equality holds in \eqref{s1:Heintze-Karcher-original} if and only if $\Sigma$ is umbilic.
\end{thmA}
The proof of Theorem A is inspired by a direct integral method due to Heintze and Karcher \cite{Heintze-Karcher1978}, using the unit normal flow in Euclidean space. Montiel and Ros \cite{Montiel-Ros1991} also proved an analogue of the Alexandrov's theorem in hyperbolic space and in the open hemisphere. However, it remains unknown for a long time whether the Heintze-Karcher inequality \eqref{s1:Heintze-Karcher-original} can be generalized to space forms. A breakthrough in this direction was made by Brendle \cite{Bre13}, who utilized the unit normal flow with respect to a conformal metric to establish the Heintze-Karcher type inequality in certain warped product manifolds. In particular, Brendle proved the following result.
\begin{thmB}\cite{Bre13}
	Let $\Omega$ be a bounded domain with smooth boundary $\Sigma=\partial \Omega$ in hyperbolic space $\mathbb H^{n+1}$. Fix a point $o\in \mathbb H^{n+1}$ and $V=\cosh r(x)$, where $r(x)=d(x,o)$ is the distance to $o$. Assume that the mean curvature $H$ of $\Sigma$ is positive, then
	\begin{align}\label{s1:Heintze-Karcher}
	\int_{\Sigma} \frac{V}{H} d\mu  \geq \frac{n+1}{n}\int_{\Omega} V d\vol.
	\end{align}
	Equality holds in \eqref{s1:Heintze-Karcher} if and only if $\Sigma$ is umbilic.
\end{thmB}
This inequality \eqref{s1:Heintze-Karcher} was also proved later using the generalized Reilly's formula, see \cite{Qiu-Xia2015,Li-Xia2019}.

The main purpose of this paper is to establish the following new Heintze-Karcher type inequality for hypersurfaces with mean curvature $H>n$ in hyperbolic space.
\begin{thm}\label{s1:main-thm}
	Let $\Omega$ be a bounded domain with smooth boundary $\Sigma=\partial \Omega$ in hyperbolic space $\mathbb H^{n+1}$. Fix a point $o\in \mathbb H^{n+1}$ and $V(x)=\cosh r(x)$, where $r(x)=d(o,x)$ is the distance to this point $o$. Assume that the mean curvature of $\Sigma=\partial\Omega$ satisfies $H>n$, then
	\begin{align}\label{s1:shifted-Heintze-Karcher}
	\int_{\Sigma} \frac{V-V_{,\nu}}{H-n} d\mu  \geq \frac{n+1}{n}\int_{\Omega} V d\vol,
	\end{align}
	where $V_{,\nu}=\langle \-\nabla V,\nu\rangle=\langle \sinh r \partial_r,\nu \rangle$ is the support function of $\Sigma$ and $\nu$ denotes the unit outward normal of $\Sigma$. Equality holds in \eqref{s1:shifted-Heintze-Karcher} if and only if $\Sigma$ is umbilic.
\end{thm}
For the curve case (i.e., $n=1$), the inequality \eqref{s1:shifted-Heintze-Karcher} was proved by Li and Xu \cite[Proposition 8.2]{Li-Xu2022}. However, it seems difficult to generalize their proof to the hypersurfaces in $\mathbb H^{n+1}$ $(n\geq 2)$. Our proof of Theorem \ref{s1:main-thm} relies on the unit normal flow, whose solution is a family of level sets $\{v=t\}$. Here $v(x)=d(x,\Sigma)$ is the distance  from $x\in\Omega$ to the hypersurface $\Sigma$. Since the distance function $v$ is merely Lipschitz continuous, we define the following geometric quantity
$$
Q(t)=e^{(n+1)t}\left[\int_{\Sigma_t^\ast}\frac{V-V_{,\nu}}{H-n}d\mu-\frac{n+1}{n}\int_{\{v>t\}} V d\vol\right].
$$
where $\Sigma_t^\ast$ is the smooth part of $\{v=t\}$ and $\{v>t\}$ is the super-level set.  Combining with the Minkowski type formula on $\Sigma_t^\ast$ for a.e. $t\in (0,T)$ (Lemma \ref{s3:prop-identity}), we prove that $Q(0)\geq Q(t)$ for all $t\in (0,T)$ (Proposition \ref{s3:monotonicity}). Finally, Theorem \ref{s1:main-thm} follows from the asymptotic estimate $\liminf_{t\ra T} Q(t)\geq 0$.

Next, we will provide two applications of the Heintze-Karcher type inequality \eqref{s1:shifted-Heintze-Karcher} in Theorem \ref{s1:main-thm}.   For this, we recall the definition of shifted $k$th mean curvature. Given a smooth hypersurface $\Sigma$ in $\mathbb{H}^{n+1}$, we denote its principal curvatures as $\kappa=(\kappa_1,\cdots,\kappa_n)$. The shifted principal curvatures are defined as $\tilde{\kappa}=(\~\k_1,\cdots,\~\k_n)=(\k_1-1,\cdots,\k_n-1)$. Then the shifted $k$th mean curvature $E_k(\~{\k})$ of $\Sigma$ is defined as the normalized $k$th elementary symmetric function of $\tilde{\kappa}$. These definitions arise naturally in the context of horospherically convex ($h$-convex) geometry of hypersurfaces in $\mathbb{H}^{n+1}$ (see \cite{ACW2021,Esp09}).

The first application of Theorem \ref{s1:main-thm} is an Alexandrov-type theorem for smooth hypersurface with constant shifted $k$th mean curvature in hyperbolic space.
\begin{thm}\label{s1:thm-application-1}
	Let $\Sigma$ be a smooth closed and embedded hypersurface in $\mathbb{H}^{n+1}$. Let $k\in \{2,\cdots,n\}$. If the shifted $k$th mean curvature $E_k(\~{\k})$ is constant, then it is a geodesic sphere.
\end{thm}
For $k=1$, Theorem \ref{s1:thm-application-1} reduces to the original Alexandrov theorem for smooth embedded hypersurface with constant mean curvature in hyperbolic space \cite{Montiel-Ros1991}.

For the second application of Theorem \ref{s1:main-thm}, we recall the following uniqueness result for $h$-convex domains by Li and Xu \cite[Proposition 8.1]{Li-Xu2022}.
\begin{thmC}\cite{Li-Xu2022}
	Let $n\geq 2$ and $k\in\{1,\cdots, n\}$. Assume that $\chi(s)$ is a smooth, positive and monotone non-decreasing function defined on $\mathbb R^{+}$. Let $\Omega$ be a smooth uniformly h-convex bounded domain in $\mathbb H^{n+1}$ which contains the origin in its interior. If $\Sigma=\partial\Omega$ satisfies the following equation
	\begin{align*}
	E_k(\~\k)=\chi(V-V_{,\nu}),
	\end{align*}
then it must be a geodesic sphere. Moreover, if $\chi$ is strictly increasing, then it is a geodesic sphere centered at the origin.
\end{thmC}

Here we say that a smooth bounded domain in hyperbolic space is $h$-convex (resp. uniformly $h$-convex), if the principal curvatures of its boundary satisfy $\k_i\geq 1$ (resp. $\k_i\geq 1+\e$, for some constant $\e>0$). The $h$-convex domains in hyperbolic space are very natural geometric objects and have received much attentions in recent years, see e.g. \cite{BM99,CabezasRivas-Miquel2007,Wang-Xia2014,AW2018,ACW2021,Wang-Wei-Zhou2022,HL22,HLW20}.

We can use Theorem \ref{s1:main-thm} to remove the condition ``\textit{the origin lies in the interior of $\Omega$}" in the statement of Theorem C and obtain the following uniqueness result.
\begin{thm}\label{s1:thm-application-2}
	Let $n\geq 2$. Assume that $\chi(s)$ is a smooth, positive and monotone non-decreasing function defined on $\mathbb R^{+}$.  Let $\Omega$ be a smooth bounded and uniformly h-convex domain in $\mathbb{H}^{n+1}$. For each $k=1,\cdots,n$, if $\Sigma=\partial \Omega$ satisfies
	\begin{align}\label{s1.cond}
	E_k(\~\k)=\chi(V-V_{,\nu}),
	\end{align}
then $\Sigma$ is a geodesic sphere. Moreover, if $\chi$ is strictly increasing, then it is a geodesic sphere centered at the origin.
\end{thm}
We note that if the origin lies in the interior of $\Omega$, then the uniform $h$-convexity of $\Sigma=\partial\Omega$ implies that $\Sigma$ is star-shaped with respect to the origin. The star-shapedness property is crucial to establish the uniqueness result in Theorem C, see \cite[p.97]{Li-Xu2022}.

The paper is organized as follows. In Section \ref{sec:2}, we give some preliminaries on closed hypersurfaces in hyperbolic space. In Section \ref{sec:3}, we use the unit normal flow in hyperbolic space to prove Theorem \ref{s1:main-thm}. In Section \ref{sec:4}, we give the proofs of Theorems \ref{s1:thm-application-1} and \ref{s1:thm-application-2}.

\begin{ack}
   This work is supported by National Key Research and Development Program of China 2021YFA1001800 and 2020YFA0713100, and NSFC grant No.12101027.
\end{ack}

\section{Preliminaries}\label{sec:2}
In this section, we collect some preliminaries on elementary symmetric functions and geometry of hypersurfaces in hyperbolic space.

\subsection{Elementary symmetric functions}
We first review some properties of elementary symmetric functions. See \cite{Guan14} for more details.

Let $\s_m:\mathbb R^n \ra \mathbb R$ be the $m$th elementary symmetric function defined by
\begin{align*}
\s_m(\l)=\sum_{i_1<\cdots<i_m} \l_{i_1} \cdots \l_{i_m},
\end{align*}
where $\l=(\l_1,\cdots,\l_n)\in \mathbb R^n$. Similarly, the normalized $m$th elementary symmetric function is defined by
\begin{align*}
E_m(\l)=\frac{1}{\binom{n}{m}}\s_m(\l).
\end{align*}
We also set $\s_0(\l)=E_0(\l)=1$ and $\s_m(\l)=E_m(\l)=0$ for $m>n$ by convention. For a symmetric matrix $A=(A_i^j)$, we set
$$
\s_m(A)=\frac{1}{m!}\d_{j_1\cdots j_m}^{i_1\cdots i_m} A_{i_1}^{j_1}\cdots  A_{i_m}^{j_m},
$$
where $\d_{j_1\cdots j_m}^{i_1\cdots i_m}$ is the generalized Kronecker symbol.

If $\l(A)=(\l_1(A),\cdots,\l_n(A))$ are the eigenvalues of $A$, then $E_m(A)=E_m(\l(A))$. The G{\aa}rding cone is defined by
$$
\G_m^+=\{\l\in \mathbb R^n:~\s_i(\l)>0,1\leq i\leq m\}.
$$
A symmetric matrix $A$ is said to belong to $\G_m^+$ if its eigenvalues $\l(A)\in \G_m^{+}$. Denote by
\begin{align*}
(\dot{E}_{m})_i^j(A)=\frac{\partial E_m(A)}{\partial A_j^i}=\frac{m(n-m)!}{n!}\d_{i i_1\cdots i_{m-1}}^{j j_1\cdots j_{m-1}} A_{j_1}^{i_1}\cdots A_{j_{m-1}}^{i_{m-1}}.
\end{align*}
If $A$ is a diagonal matrix with $\l(A)=(\l_1,\cdots,\l_n)$, then
\begin{align*}
(\dot{E}_{m})_i^j(A)=\frac{\partial E_m}{\partial \l_i}(\l) \d_i^j.
\end{align*}
The following formulae for the elementary symmetric functions are  well-known.
\begin{lem}\cite{Hardy-Littlewood-Polya1934,Reilly1973}
	We have
	\begin{align}
	\sum_{i,j}(\dot{E}_m)_i^j(A) A_j^i=&m E_m(A),\label{s2:id-1}\\
	\sum_{i,j}(\dot{E}_m)_i^j(A) \d_j^i=&m E_{m-1}(A), \label{s2:id-2}\\
	\sum_{i,j}(\dot{E}_m)_i^j(A) (A^2)_j^i=& n E_1(A)E_m(A)-(n-m)E_{m+1}(A),\label{s2:id-3}
	\end{align}
	where $(A^2)_j^i=\sum_{k}A_j^kA_k^i$.
\end{lem}

\begin{lem}\cite[Lemma 2.5]{Guan14}
	If $\l\in \G_{m}^+$, then the following Newton-MacLaurin inequality holds
	\begin{align}\label{s2:NM-ineq-1}
	E_m(\l) \leq E_1(\l)E_{m-1}(\l). 
	\end{align}
	Equality holds if and only if $\l=cI$ for some constant $c>0$, where $I=(1,\cdots,1)$.
\end{lem}

\subsection{Hypersurfaces in hyperbolic space}

Fix a point $o\in \mathbb H^{n+1}$, we take $r(x)=d(o,x)$. Using the geodesic polar coordinates with respect to $o$, the hyperbolic space $\mathbb H^{n+1}$ can be expressed as a warped product $[0,\infty)\times \mathbb S^n$ equipped with the metric
\begin{equation*}
  \-g=dr^2+\l(r)^2 g_{\mathbb S^n},
\end{equation*}
 where $g_{\mathbb S^n}$ is the round metric of $\mathbb S^n$ and $\l(r)=\sinh r$. Then $\l'(r)=\cosh r$.

Let $\Sigma$ be a smooth hypersurface in $\mathbb H^{n+1}$ with unit outward normal $\nu$.  Let $e_1,\cdots,e_n$ be an orthonormal basis of the tangent space of $\Sigma$. Denote $\overline{\nabla}$ as the Levi-Civita connection on $\mathbb{H}^{n+1}$. Then the induced metric $g$ of $\Sigma$ is $g_{ij}=\-g(e_i,e_j)$, and the second fundamental form $h=(h_{ij})$ of $\Sigma$ in $\mathbb H^{n+1}$ is given by
\begin{align*}
h_{ij}=h(e_i,e_j)=\bar{g}(\-\nabla_{e_i} \nu,e_j).
\end{align*}
The principal curvatures $\k=(\k_1,\cdots,\k_n)$ of $\Sigma$ are the eigenvalues of the Weingarten matrix $(h_i^j)=(g^{jk}h_{ki})$, where $(g^{ij})$ is the inverse matrix of $(g_{ij})$. The mean curvature of $\Sigma$ is defined as
\begin{equation*}
  H=g^{ij}h_{ij}=\sum_{i=1}^{n}\kappa_i.
\end{equation*}
 Without loss of generality, we assume that $e_1,\cdots,e_n$ are the corresponding principal directions with respect to the principal curvatures $\k_1,\cdots,\k_n$.

\begin{lem}
	Let $\Sigma$ be a smooth hypersurface in $\mathbb H^{n+1}$. Denote $V=\cosh r$ and $V_{,\nu}=\langle \-\nabla V,\nu\rangle=\langle \l\partial_r,\nu\rangle$. Then we have
    \begin{align}\label{s2.Dv}
		\nabla_i V =\langle \-\nabla V,e_i\rangle=\langle \l(r)\partial_r,e_i\rangle, \quad \nabla_i V_{,\nu}=\langle \l\partial_r, h_{i}^je_j\rangle,
    \end{align}
    and
    \begin{align}
    & \nabla_i\nabla_j V=V g_{ij}-V_{,\nu} h_{ij},\label{s2.D2V}\\
   & \nabla_i\nabla_j V_{,\nu}=\langle \l\partial_r,\nabla h_{ij}\rangle+V h_{ij} - V_{,\nu}(h^2)_{ij},\label{s3:hess-supp}
    \end{align}
    where $\{e_1,\cdots,e_n\}$ is a local orthonormal frame on $\Sigma$ and $(h^2)_{ij}=h_i^k h_{kj}$.
	Furthermore, if the hypersurface $\Sigma$ is closed, then for each $\e\in \mathbb R$, there holds
	\begin{align}\label{s2:shift-MF}
	\int_{\Sigma} (V-\e V_{,\nu}) E_{k-1}(A) d\mu = \int_{\Sigma} V_{,\nu} E_{k}(A) d\mu, \quad k=1,\cdots,n,
	\end{align}
	where $A=(h_j^i-\e \d_j^i)$.
\end{lem}
\begin{proof}
    \eqref{s2.Dv} -- \eqref{s3:hess-supp} are well-known, see e.g. \cite{Guan-Li2014}. To prove \eqref{s2:shift-MF}, by \eqref{s2:id-1} -- \eqref{s2:id-2} and \eqref{s2.D2V}, we have
	\begin{align}\label{s2.div}
	 \sum_{i,j}(\dot{E}_k)_i^j(A) \nabla^i\nabla_j V=&\sum_{i,j}(\dot{E}_k)_i^j(A) (V\d_j^i- V_{,\nu} h_j^i)\nonumber\\
=&\sum_{i,j}(\dot{E}_k)_i^j(A) \left[ (V-\e V_{,\nu})\d_j^i-V_{,\nu}(h_j^i-\e \d_j^i)\right]\nonumber\\
=&k\Big((V-\e V_{,\nu}) E_{k-1}(A)-V_{,\nu}E_{k}(A)\Big).
	\end{align}
 Note that $A_{ij}=h_{ij}-\e g_{ij}$ is a Codazzi tensor, i.e., $\nabla_\ell A_{ij}$ is symmetric in $i,j,\ell$. It follows that $(\dot{E}_k)_i^j(A)$ is divergence free. The equation \eqref{s2:shift-MF} follows from  integration by parts and the divergence free property of $(\dot{E}_k)_i^j(A)$.
\end{proof}

Along the general flow
\begin{align}\label{s2:general-flow}
\frac{\partial}{\partial t}X=\mathcal{F}\nu
\end{align}
for hypersurfaces in hyperbolic space $\mathbb{H}^{n+1}$, where $\nu$ is the unit outward normal and $\mathcal{F}$ is the speed function of the smooth part of the evolving hypersurface $\Sigma_t$, we have the following evolution equations (see e.g. \cite[Theorem 3.2]{HP99}).
\begin{prop}\label{s2:prop-general-evolution}
	Along the flow \eqref{s2:general-flow} in $\mathbb{H}^{n+1}$, we have
	\begin{align}
	\partial_t g_{ij}=&2\mathcal{F} h_{ij}, \label{s2:evolution-induced-metric}\\
	\partial_t d\mu_t=&\mathcal{F}H d\mu_t, \label{s2:evolution-area}\\
	\partial_t \nu =&-\nabla \mathcal{F}, \label{s2:evolution-unit-normal}\\
	\partial_t h_i^j=&-\nabla^j \nabla_i \mathcal{F}-\mathcal{F}((h^2)_i^j-\d_i^j),\label{s2:evolution-Weingarten}
	\end{align}
	where $(h^2)_i^j=h_i^k h_k^j$, and $d\mu_t$ is the area form on $\Sigma_t^\ast$.
\end{prop}

\section{Proof of Theorem \ref{s1:main-thm}}\label{sec:3}
Let $\Omega$ be a bounded domain in $(\mathbb H^{n+1},\-g)$ with smooth boundary $\Sigma=\partial \Omega$. Let $\nu$ denote the outward unit normal to $\Sigma$. Throughout this section, we assume that the mean curvature of $\Sigma$ satisfies $H>n$ with respect to this choice of unit normal.

Denote by $X:\Sigma\times [0,\infty)\ra \-\Omega$ the normal exponential map with respect to $\-g$. For each point $y\in \Sigma$, the curve $t \mapsto X(y,t)$ is a geodesic which has unit speed, and
\begin{align}\label{s2:unit-normal-flow}
\left\{\begin{aligned}
\left.\frac{\partial}{\partial t}X(y,t)\right|_{t=0}=&-\nu(y),\\
X(y,0)=&y.
\end{aligned}\right.
\end{align}
For each point $p\in \-\Omega$, we denote by $v(p)=d(p,\Sigma)$ the geodesic distance to $\Sigma$ and $\{v=0\}=\Sigma$.  We define
\begin{align*}
A=&\left\{(x,t)\in \Sigma\times [0,\infty):v(X(x,t))=t\right\},
\end{align*}
and
\begin{align*}
A^\ast=&\left\{(x,t)\in \Sigma\times [0,\infty):(x,t+\d)\in A ~\text{for some $\d>0$}\right\}.
\end{align*}
The following proposition follows from the definition of segment domain of a Riemannian manifold, see \cite[Proposition 3.1]{Bre13}.
\begin{prop}\cite{Bre13}
	The set $A$ and $A^\ast$ have the following properties:
	\begin{enumerate}[$(i)$]
		\item If $(x,t_0)\in A$, then $(x,t)\in A$ for all $t\in [0,t_0]$.
		\item The set $A$ is closed, and we have $X(A)=\-\Omega$.
		\item The set $A^\ast$ is an open subset of $\Sigma\times [0,\infty)$, and the restriction $X|_{A^\ast}$ is a diffeomorphism.
	\end{enumerate}
\end{prop}
The set $G=X(A^\ast)$ is the largest open subset of $\Omega$ such that for every $x$ in $G$, there is a unique closest point $y\in \Sigma$ to $x$. Moreover, the boundary $\Sigma$ is smooth, and thus the distance function $d(\cdot,\Sigma):\-\Omega \ra [0,\infty)$ belongs to $C^{\infty}(G\cup \Sigma)\cap \operatorname{Lip}(\Omega)$. The set $C=\Omega\backslash G$ is called the {\em cut locus} of $\Sigma$. A result of Li and Nirenberg \cite[Corollary 1.6]{Li-Nirenberg2005} states that the $n$-dimensional Hausdorff measure of the cut-locus $C$ of the hypersurface $\Sigma$ is finite, i.e., $\mathcal{H}^{n}(C)<\infty$.

Consider a geodesic from a point $y\in \Sigma$ going into $\Omega$ in the normal direction until it first hits a point in $x=m(y)$ in $C$. The point $m(y)$ is called the {\em cut point of $y$} on $\Sigma$, which means that if we go beyond the point $x$ on the geodesic to any other point $x'$, then $x'$ has a closer point $y'\in\Sigma$ such that $d(x',y')<d(x',y)$. The cut locus of $\Sigma$ is the collection of all these cut points $m(y)$ for all $y\in \Sigma$. An equivalent way to define the cut point of $y\in \Sigma$ is as follows: The set of $t>0$ satisfying $d(X(y,t),\Sigma)=t$ is either $(0,\infty)$ or $(0,\~t(y)]$ for some $0<\~t(y)<\infty$. In the latter case, $\~m(y)=X(y,\~t(y))$ is the cut point of $y\in \Sigma$, and the collection of $\~m(y)$ for all $y\in \Sigma$, denoted as $\~C$, is also called the cut locus of $\Sigma$. These two definitions are the same, i.e., $\~m(y)=m(y)$ for all $y\in \Sigma$ and $\~C=C$, see \cite[Section 4]{Li-Nirenberg2005}.

According to \cite[Theorem 1.5]{Li-Nirenberg2005}, as $\Sigma$ is compact, the function $d(y,m(y))$ is Lipschitz continuous in $y\in \Sigma$, and hence $\~t(y)=d(y,X(y,\~t(y)))=d(y,\~m(y))=d(y,m(y))$ is also Lipschitz continuous in $y\in \Sigma$. Let $T=\max_{y\in \Sigma}\{\~t(y)\}<\infty$. Then we have $0\leq v(x)\leq T<\infty$ for all $x\in \-\Omega$. For each $t\in [0,T]$, we define
\begin{align*}
\Sigma_t^\ast=X(A^\ast \cap (\Sigma \times \{t\})).
\end{align*}
Note that $\Sigma_t^\ast$ is a smooth hypersurface which is contained in the level set $\{v=t\}$. Then we have  $\Sigma_0^\ast=\Sigma=\{v=0\}$, and $\Sigma_T^{\ast}=\emptyset$ due to the definition of $A^\ast$. In particular, $\Sigma_t^\ast=\{v=t\} \backslash C$ for all $0<t<T$ and
$$
\Omega\backslash C=X(A^\ast)=\bigcup_{0< t<T}\Sigma_t^\ast.
$$

Denote by $H$ and $h$ the mean curvature and the second fundamental form of the smooth hypersurface $\Sigma_t^\ast$ in hyperbolic space.
\begin{lem}\label{s3:prop-differential-ineq}
	The mean curvature of $\Sigma_t^\ast$ satisfies $H>n$ and the following differential inequality
	\begin{align}\label{s2:evol-quantity}
	\frac{\partial}{\partial t}\(\frac{V-V_{,\nu}}{H-n}\)\leq -\frac{H}{n}\frac{V-V_{,\nu}}{H-n}.
	\end{align}
\end{lem}
\begin{proof}
    Using Proposition \ref{s2:prop-general-evolution} with $\mathcal F=-1$, the mean curvature of $\Sigma_t^\ast$ evolves by
	\begin{align*}
	\frac{\partial}{\partial t}H=|h|^2-n.
	\end{align*}
	By the trace inequality $|h|^2\geq H^2/n$, we obtain
	\begin{align}\label{s3.pf1}
	\frac{\partial}{\partial t}(H-n)\geq \frac{1}{n}(H-n)(H+n)
	\end{align}
	at each point on $\Sigma_t^\ast$. Since the mean curvature of the initial hypersurface $\Sigma$ satisfies $H>n$, applying the comparison principle to \eqref{s3.pf1} we conclude that $H>n$ on $\Sigma_t^{\ast}$ for each $t\in [0,T)$.
	
    To prove the differential inequality \eqref{s2:evol-quantity}, we use the evolution equation \eqref{s2:evolution-unit-normal} of the unit normal $\nu$ to calculate that
	\begin{align}
	\frac{\partial}{\partial t}V=&\langle \-\nabla V,-\nu\rangle=-V_{,\nu}, \label{s3:eq-V-1}\\
	\frac{\partial}{\partial t}V_{,\nu}=&\langle \-\nabla_{-\nu}(\l \partial_r), \nu\rangle+\langle \l\partial_r,\partial_t\nu\rangle=-\l'=-V,\label{s3:eq-V-2}
	\end{align}
	where we used the fact that $\l\partial_r$ is a conformal Killing vector field, that is,
\begin{equation}\label{s3.conf}
  \langle\overline{\nabla}_X(\lambda\partial_r),Y\rangle=\lambda'\langle X,Y\rangle
\end{equation}
for any tangent vector fields $X,Y$ of $\mathbb{H}^{n+1}$ (see \cite[Lemma 2.2]{Bre13}). Since
\begin{align*}
V-V_{,\nu}=\cosh r-\sinh r\langle \partial_r,\nu\rangle\geq \cosh r-\sinh r>0,
\end{align*}
it follows from \eqref{s3.pf1}--\eqref{s3:eq-V-2} that
	\begin{align*}
	\frac{\partial}{\partial t}\(\frac{V-V_{,\nu}}{H-n}\)=&\frac{1}{H-n}\frac{\partial}{\partial t}(V-V_{,\nu})-\frac{V-V_{,\nu}}{(H-n)^2}\frac{\partial}{\partial t}H \\
\leq &-\frac{H}{n}\frac{V-V_{,\nu}}{H-n}
	\end{align*}
    at each point on $\Sigma_t^\ast$. This completes the proof.	
\end{proof}

\begin{cor}\label{s3:cor-area-decreasing}
	The function $t\mapsto \mu(\Sigma_t^\ast)$ is monotone decreasing.
\end{cor}
\begin{proof}
	On $\Sigma_t^\ast$, the area form $d\mu_t$ satisfies
	\begin{align}\label{s2:evol-area-form}
	\frac{\partial}{\partial t} d\mu_t=-H d\mu_t  \quad \text{on $\Sigma_t^\ast$}.
	\end{align}	
	Since $H>n$ on $\Sigma_t^\ast$, \eqref{s2:evol-area-form} implies that the area form of $\Sigma_t^\ast$ is monotone decreasing in $t$. Moreover, the set $\{x\in \Sigma:(x,t)\in A^\ast\}$ becomes smaller as $t$ increases. Then the assertion follows.
\end{proof}

The following lemma is crucial in the proof of Theorem \ref{s1:main-thm}.
\begin{lem}\label{s3:prop-identity}
	For a.e. $t\in (0,T)$, there holds
	\begin{align}\label{s3.eq3-4}
	\int_{\Sigma_t^\ast} V_{,\nu} d\mu=(n+1)\int_{\{v>t\}} V d\vol.
	\end{align}
\end{lem}
\begin{proof}
	Let $(\mathbb B^{n+1},\d)$ be the unit open ball centered at the origin in $\mathbb R^{n+1}$ equipped with the Euclidean metric $\d$. We use the Poincar\'e ball model to represent the hyperbolic space $\mathbb H^{n+1}$, that is,
	\begin{align*}
	(\mathbb H^{n+1},\-g)=\(\mathbb B^{n+1},f^2\d\), \quad f^2=\dfrac{4}{(1-|x|_\d^2)^2},
	\end{align*}
	where $|x|_\d=\tanh\(\frac{r}{2}\)\in [0,1)$ and $r=r(x)$ is the hyperbolic distance to the origin $o$. It is easy to check that $f=\cosh r+1$.

Let $\Omega$ be the bounded domain in $\mathbb H^{n+1}$ with smooth boundary $\Sigma$, then it can be also viewed as a smooth bounded domain in $\mathbb B^{n+1}$. For each $0<t<T$, the level set $\{v=t\}$, which is the topological boundary of $\{v>t\}$, is a solution of the unit normal flow \eqref{s2:unit-normal-flow}.
	
	{\bf Claim}: For a.e. $t\in (0,T)$, there holds
	\begin{align}\label{s3:t-condition-1}
	\mathcal{H}^n(\{v=t\}\backslash \Sigma_t^\ast)=0.
	\end{align}
	\begin{proof}[Proof of Claim] It follows from \cite[Corollary 1.6]{Li-Nirenberg2005} that $\mathcal{H}^n(C)<\infty$. Suppose \eqref{s3:t-condition-1} does not hold, there exists uncountably many $\{t_\alpha\}_{\a\in \Lambda}\in (0,T)$ such that
	$\mathcal{H}^n(\{v=t\}\backslash \Sigma_{t_\alpha}^\ast)>0$. It follows that
	\begin{align*}
	\mathcal{H}^n(C)=&\mathcal{H}^n\(\bigcup_{0<t<T}(\{v=t\}\backslash\Sigma_t^\ast)\)\\
\geq &\sum_{\a\in \Lambda}\mathcal{H}^n(\{v=t\}\backslash \Sigma_{t_\alpha}^\ast)=\infty,
	\end{align*}
	which contradicts with $\mathcal{H}^n(C)<\infty$.
\end{proof}
	
	We continue the proof of Lemma \ref{s3:prop-identity}. Since $v$ is Lipschitz continuous, by \cite[Theorem 13.1 and Example 13.3]{Maggi2012} we conclude that for a.e. $t\in (0,T)$, the super-level set $\{v>t\}$ is an open set of finite perimeter in $\mathbb R^{n+1}$. The reduced boundary $\partial^\ast \{v>t\}$ of $\{v>t\}$ exists and it is always a subset of the topological boundary $\{v=t\}$, see \cite[Remark 15.3]{Maggi2012}. Moreover, by the coarea formula for Lipschitz functions, we have
	\begin{align}\label{s3:t-condition-2}
	\mathcal{H}^n(\partial^\ast \{v>t\})=\mathcal{H}^n(\{v=t\}), \quad \text{for a.e. $t\in [0,T)$},
	\end{align}
	see \cite[Theorem 18.1 and Remark 18.2]{Maggi2012}. Recall that the reduced boudary $\partial^\ast \{v>t\}$ is the union of countably many compact subsets $K_i$ of $C^1$-hypersurfaces such that $\partial^\ast \{v>t\}=N \cup \bigcup_{i\geq 1}K_i$, where $\mathcal{H}^n(N)=0$. It follows from \eqref{s3:t-condition-1} and \eqref{s3:t-condition-2} that for a.e. $t\in (0,T)$,
	$\{v=t\}=\Sigma_t^\ast \cup Z$ with $\mathcal{H}^n(Z)=0$ and hence $\{v=t\}$ is an open set with almost $C^1$-boundary.
	
	Let $\{E_i\}_{i=1}^{n+1}$ be an orthonormal basis in $(\mathbb B^{n+1},\d)$. Then $\{\-E_i=f^{-1}E_i\}_{i=1}^{n+1}$ is an orthonormal basis of $(\mathbb B^{n+1},f^2\d)$. Denote by $\-\nabla$ (resp. $\nabla^\d$) and $\-\D$ (resp. $\D_\d$) the gradient and the Laplacian with respect to $\-g=f^2\d$ (resp. $\d$). For any function $F\in C^\infty(\mathbb B^{n+1},f^2\d)$, it can also be considered as a function $F\in C^\infty(\mathbb B^{n+1},\d)$. Then there holds
	\begin{align*}
	\nabla^\d F=\sum_{i=1}^{n+1}E_i(F)E_i=f^2 \sum_{i=1}^{n+1}\-E_i(F)\-E_i=f^2 \-\nabla F.
	\end{align*}
	Using the transformation law of the Laplacian for $\-g=f^2 \d$, one has
	\begin{align}\label{s3:conformal-Laplacian}
	\-\D V=f^{-2}\(\D_\d V+(n-1)d\ln f(\nabla^\d V)\).
	\end{align}	
	We also have
	\begin{align}\label{s3:conformal-relation}
	d\vol=f^{n+1}d\vol_{\d}, \quad d\mu=f^n d\mu_\d, \quad \nu=f^{-1} \nu_\d,
	\end{align}
	where $d\vol_\d$ is Euclidean volume form, $d\mu_\d$ is the area form and $\nu_\d$ is the unit outward normal of $\Sigma$ in $(\mathbb B^{n+1},\d)$.

Using the Gauss-Green theorem on open sets with almost $C^1$-boundary in Euclidean space \cite[Theorem 9.6]{Maggi2012}, there holds
	\begin{align}\label{s3:Gauss-Green-formula}
	\int_{\Sigma_t^\ast} f^{n-1} dV(\nu_\d) d\mu_\d=&\int_{\{v>t\}} \divv^\d (f^{n-1}\nabla^\d V) d\vol_\d \nonumber\\
	=&\int_{\{v>t\}} \(f^{n-1} \D_\d V+(n-1)f^{n-2} df(\nabla^\d V)\) d\vol_\d.
	\end{align}
Applying \eqref{s3:conformal-Laplacian} and \eqref{s3:conformal-relation} to \eqref{s3:Gauss-Green-formula}, we obtain
	\begin{align*}
	\int_{\{v>t\}}\-\D V d\vol=&\int_{\{v>t\}} \(f^{n-1}\D_\d V+(n-1)f^{n-2}df(\nabla^\d V) \)d\vol_\d\\
	=&\int_{\Sigma_t^\ast} f^{n-1}dV(\nu_\d) d\mu_\d\\
	=&\int_{\Sigma_t^\ast} \langle \-\nabla V,\nu\rangle d\mu.
	\end{align*}
Since $\-\D V=(n+1)V$ (which follows from the conformal property \eqref{s3.conf}, see also \cite{Montiel-Ros1991}), we obtain the equation \eqref{s3.eq3-4}.
\end{proof}

Next, we introduce the quantity
\begin{align}\label{s3:def-Q(t)}
Q(t)=e^{(n+1)t}\left[\int_{\Sigma_t^\ast}\frac{V-V_{,\nu}}{H-n}d\mu-\frac{n+1}{n}\int_{\{v>t\}} V d\vol\right].
\end{align}
We have the following monotonicity property.
\begin{prop}\label{s3:monotonicity}
	For any $t\in [0,T)$, there holds
	\begin{align*}
	Q(0)\geq Q(t).
	\end{align*}
\end{prop}
\begin{proof}
	By the coarea formula, we have
	\begin{align*}
	\int_{\{v>t\}} V d\vol=\int_{X(A^\ast \cap (\Sigma \times (t,T)))} V d\vol=\int_t^{T} \int_{\Sigma_\tau^\ast} V d\mu d\tau,
	\end{align*}
	and hence
	\begin{align}\label{s2:evol-weighted-volume}
	\frac{d}{dt}\int_{\{v>t\}} V d\vol=-\int_{\Sigma_t^\ast} V d\mu_t.
	\end{align}
	Using \eqref{s2:evol-quantity}, \eqref{s2:evol-area-form} and \eqref{s2:evol-weighted-volume}, we obtain
	\begin{align}\label{s3.pf2}
	&\limsup_{h \searrow 0}\frac{1}{h}\(Q(t)-Q(t-h)\)\nonumber\\
	\leq &(n+1)Q(t)+e^{(n+1)t}\Bigg(\int_{\Sigma_t^\ast}\frac{\partial}{\partial t}\(\frac{V-V_{,\nu}}{H-n}\)d\mu_t\nonumber\\
&\qquad -\int_{\Sigma_t^\ast}\frac{V-V_{,\nu}}{H-n}Hd\mu_t+\frac{n+1}{n}\int_{\Sigma_t^\ast} V d\mu_t\Bigg)\nonumber\\
	\leq &\frac{n+1}{n}e^{(n+1)t}\left(\int_{\Sigma_t^\ast}V_{,\nu}d\mu_t-(n+1)\int_{\{v>t\}}Vd\vol\right).
	\end{align}
	By Lemma \ref{s3:prop-identity}, the last line of \eqref{s3.pf2} vanishes for a.e. $t\in [0,T)$. It follows that for any $t\in [0,T)$,
\begin{align*}
  Q(t)-Q(0)\leq &\int_0^t\frac{n+1}{n}e^{(n+1)\tau}\left(\int_{\Sigma_\tau^\ast}V_{,\nu}d\mu_t-(n+1)\int_{\{v>\tau\}}Vd\vol\right)d\tau\\
  \leq &0.
\end{align*}
\end{proof}

Now we can complete the proof of Theorem \ref{s1:main-thm}.
\begin{proof}[Proof of Theorem \ref{s1:main-thm}]
	By Proposition \ref{s3:monotonicity}, we have
	\begin{align}\label{s3:comparison}
	\int_{\Sigma}\frac{V-V_{,\nu}}{H-n}d\mu-\frac{n+1}{n}\int_{\Omega}V d\vol=Q(0) \geq \liminf_{t\ra T} Q(t),
	\end{align}
    where we used the fact that $\Sigma=\Sigma_0^\ast$ and $\Omega=\{v>0\}$. To prove the desired inequality \eqref{s1:shifted-Heintze-Karcher}, it suffices to show that $\liminf_{t\ra T} Q(t)\geq 0$.

    Since $V-V_{,\nu}>0$ and $H>n$ on $\Sigma_t^\ast$ (by Lemma \ref{s3:prop-differential-ineq}), we deduce that for a.e. $t\in [0,T)$,
	\begin{align*}
	Q(t)\geq& -\frac{n+1}{n}e^{(n+1)T}\int_{\{v>t\}} V d\vol \\
	=&-\frac{n+1}{n}e^{(n+1)T}\int_{X(A^\ast\cap (\Sigma \times (t,T))) }V d\vol\\
	= &-\frac{n+1}{n}e^{(n+1)T}\int_{t}^{T} \int_{\Sigma_\tau^\ast}V d\mu d\tau.
	\end{align*}
	As $d\mu_t$ is monotone decreasing by \eqref{s2:evol-area-form}, $V=\cosh r$ is uniformly bounded on $\Omega$ and $T<\infty$, we conclude that
\begin{equation*}
  \liminf_{t\ra T}Q(t)\geq 0.
\end{equation*}

    If equality holds in \eqref{s1:shifted-Heintze-Karcher}, then the equality holds in the differential inequality \eqref{s2:evol-quantity}, which implies that $\Sigma_t^\ast$ is totally umbilical for all $t\in [0,T)$. Therefore, the initial hypersurface $\Sigma=\Sigma_0^\ast$ is umbilical. This completes the proof of Theorem \ref{s1:main-thm}.
\end{proof}

\section{Proof of Theorems \ref{s1:thm-application-1} and \ref{s1:thm-application-2}}\label{sec:4}
\begin{proof}[Proof of Theorem \ref{s1:thm-application-1}]
	Fix $k\in \{2,\cdots,n\}$. Since $\Sigma$ is compact, there exists an elliptic point such that all principal curvatures $\k_i>1$ for $i=1,\cdots,n$ and hence the constant $E_k(\~{\k})$ is positive. By continuity of the principal curvatures, we have $\~{\k}\in \G_k^{+}$. Then we have
	\begin{align}\label{s4.pf1}
	\int_{\Sigma} (V-V_{,\nu}) \frac{E_{k-1}(\~{\k})}{E_k(\~{\k})}d\mu=&\frac{1}{E_k(\~{\k})}\int_{\Sigma} (V-V_{,\nu}) E_{k-1}(\~{\k})d\mu\nonumber\\
=&\frac{1}{E_k(\~{\k})}\int_{\Sigma} V_{,\nu} E_k(\~{\k})d\mu\nonumber\\
	=&\int_{\Sigma} V_{,\nu} d\mu,
	\end{align}
	where in the second equality we used Minkowski identity \eqref{s2:shift-MF}.
	By the Minkowski formula (see \cite{Montiel-Ros1991}), we also have
\begin{equation}\label{s4.Mink}
\int_{\Sigma}V_{,\nu}d\mu=(n+1)\int_\Omega V d\mathrm{vol},
\end{equation}
where $\Omega$ is the domain enclosed by $\Sigma$. Then the Heintze-Karcher type inequality \eqref{s1:shifted-Heintze-Karcher} and the Newton-MacLaurin inequality \eqref{s2:NM-ineq-1} imply
	\begin{align}\label{s4.pf2}
	\int_{\Sigma} (V-V_{,\nu}) \frac{E_{k-1}(\~{\k})}{E_k(\~{\k})}d\mu\geq& \int_{\Sigma} \frac{V-V_{,\nu}}{E_1(\~{\k})}d\mu\nonumber\\
\geq & (n+1)\int_\Omega V d\mathrm{vol}\nonumber\\
=&\int_{\Sigma}V_{,\nu}d\mu.
	\end{align}
	Combining \eqref{s4.pf1} and \eqref{s4.pf2}, we conclude that the equality holds in the Newton-MacLaurin inequality \eqref{s2:NM-ineq-1}, so $\Sigma$ is totally umbilical and it is a geodesic sphere.
\end{proof}

\begin{proof}[Proof of Theorem \ref{s1:thm-application-2}]
	As the hypersurface $\Sigma$ is uniformly $h$-convex, we have $\~{\k}_i>0$, $E_k(\~\k)>0$ and $\frac{\partial E_k(\~{\k})}{\partial \~\k_i}>0$. We deduce from the Heintze-Karcher type inequality \eqref{s1:shifted-Heintze-Karcher} and \eqref{s4.Mink} that
	\begin{align}\label{s4.1}
	0\leq &\int_{\Sigma} \(\frac{V-V_{,\nu}}{E_1(\~{\k})}-V_{,\nu}\)d\mu\nonumber\\
	\leq &\int_{\Sigma} \frac{(V-V_{,\nu})E_{k-1}(\~{\k})-V_{,\nu} E_{k}(\~{\k})}{E_k(\~{\k})}d\mu\nonumber\\
	=&\int_{\Sigma}\frac{1}{\chi}\Big((V-V_{,\nu})E_{k-1}(\~{\k})-V_{,\nu} E_{k}(\~{\k})\Big)d\mu,
	\end{align}
	where the second inequality follows from $V>V_{,\nu}$ and the Newton-MacLaurin inequality \eqref{s2:NM-ineq-1}, and the last equality follows from the curvature equation \eqref{s1.cond}. Using \eqref{s2.div} with $\varepsilon=1$, we can rewrite the last line of \eqref{s4.1} as
\begin{align}\label{s4.2}
&\int_{\Sigma}\frac{1}{\chi}\Big((V-V_{,\nu})E_{k-1}(\~{\k})-V_{,\nu} E_{k}(\~{\k})\Big)d\mu\nonumber\\
=&\frac{1}{k}\int_{\Sigma}\frac{1}{\chi}\sum_{i=1}^n \frac{\partial E_k(\~{\k})}{\partial\~\k_i}\nabla_i\nabla_iVd\mu\nonumber\\
	=&\frac{1}{k}\int_{\Sigma} \frac{\chi'}{\chi^2}\sum_{i=1}^n \frac{\partial E_k(\~{\k})}{\partial \~\k_i} \nabla_i V\nabla_i(V-V_{,\nu})d\mu\nonumber\\
	=&-\frac{1}{k}\int_{\Sigma} \frac{\chi'}{\chi^2}\sum_{i=1}^n \frac{\partial E_k(\~{\k})}{\partial \~\k_i}\~{\k}_i |\nabla_i V|^2d\mu\leq 0,
	\end{align}
	where we used integration by parts and the divergence-free property of $E_k(\~\k)$ in the second equality, and the third equality follows from \eqref{s2.Dv}. The last inequality is due to $\chi>0$ and $\chi'\geq 0$.

Combining \eqref{s4.1} and \eqref{s4.2}, we see that both the Heintze-Karcher type inequality \eqref{s1:shifted-Heintze-Karcher} and Newton-MacLaurin inequality \eqref{s2:NM-ineq-1} must achieve the equality. This implies that $\Sigma$ is a geodesic sphere. Furthermore, if $\chi'>0$, the last inequality of \eqref{s4.2} implies that $\nabla_iV=0$ for all $i=1,\cdots,n$. Then the unit normal $\nu$ is proportional to the radial direction $\partial_r$, and hence $\Sigma$ must be a geodesic sphere centered at the origin.	
\end{proof}

\end{document}